\documentclass[12pt,a4paper]{article}
\usepackage{graphicx}
\usepackage{amsmath,amscd,amssymb,latexsym,upref} 
\usepackage{color,enumitem}
\usepackage{bm}

\newcommand{\Hil}[0]{
\mathcal{H} 
}
\newcommand{\norm}[2]{
\left\| #2 \right\|_{#1}
}

\newcommand{\BL}[0]{
{\mathcal B}
}

\newtheorem{theorem}{Theorem}[section]

\newtheorem{proposition}[theorem]{Proposition}
\newtheorem{lemma}[theorem]{Lemma}

\newtheorem{conj.}[theorem]{Conjecture}
\newtheorem{Bsp.}{Example}[section]
\newtheorem{remark}{Remark}[section]
\newenvironment{proof}{\noindent \bf Proof: \rm}{$ \hspace{\stretch{1}} \Box $

\vspace{5mm}}
\newcommand{\Op}{\mathcal O}
\newcommand{\Mat}{\mathcal M}
\newcommand{\TOp}{\tilde{\mathcal O}}
\newcommand{\range}[1]{\mathsf{ran}\left( #1 \right)} 
\newcommand{\identity}[1]{\mathsf{id}_{ #1 }}
\newcommand{\kernel}[1]{\mathsf{ker}\left( #1 \right)}

\title{Redundant Representation of Operators}
\author{Peter Balazs$^*$ and Georg Rieckh$^\dagger$}

\begin{document}

\maketitle

\thanks{${}^*$  Austrian Academy of Sciences, Acoustics Research Institute, Reichsratsstrasse 17, A-1010 Vienna, Austria, Peter.Balazs@oeaw.ac.at; \\ ${}^\dagger$ Information Processing in Biological Networks, Institute of Science and Technology, Am Campus 1, A-3400 Klosterneuburg, Austria}

\begin{abstract} To be able to solve operator equations numerically a discretization of those operators is necessary. 
In the Galerkin approach bases are used to achieve discretized versions of operators.
In a more general set-up, frames can be used to sample the involved signal spaces and therefore those operators.
Here we look at the redundant representation of operators resulting from a matrix representation using frames.
We focus on injectivity, surjectivity and, in particular, invertibility of the involved operators and matrices.
Furthermore we show sufficient conditions that the composition of matrices correspond to the composition of operators.
\\[2mm]
\end{abstract}

\section{Introduction}

In applied mathematics, one often has to solve operator equations numerically.
In computational acoustics, for example, this is done to analyze sound fields and vibrations.
Here the finite element \cite{hack03} and the boundary element method \cite{sausch04} are widely used.
One particular scheme to discretize the operator equations is the Galerkin method \cite{Gauletal03}.
This corresponds to taking finite sections of the standard matrix description \cite{gohberg1} of operators $O$ using an ONB (or biorthogonal basis) $(e_k)$.
The corresponding matrix $M$ is then constructed by calculating its entries $M_{j,k} = \langle O e_k, e_j\rangle$. 
Depending on the operator, the $(e_k)$ in the above discretization scheme can be chosen such that the resulting matrix has certain desired properties.
These, in turn, can lead to a more efficient numerical solution.
If, however, one is restricted to bases, the search for $(e_k)$ with advantageous properties can prove to be difficult.

The generalization to frames \cite{casaz1,ole1} can relax these constrains.
Because of their (possible) over-completeness they have applications in signal processing and related fields.
Recently the representation of operators using frames has received some attention \cite{
aria07,xxlframoper1,rud11}. 
Certain operators, named multipliers, which have a diagonal matrix representation are of special interest in mathematics \cite{Arias2008581,xxlmult1,xxlframehs07} as well as acoustical applications \cite{xxllabmask1,DepKronTor07,majxxl10,tokr13}.
In this case the invertibility of such operators is the topic of current research \cite{balsto09new,uncconv2011}. 
Interestingly those kind of operators also play an important role as quantization operators \cite{aliant1,colgaz10,xxlbayasg11}. 
Multipliers can be also used to find a diagonalization of operators using frames, see \cite{Futamura20123201}.\\

In this paper we extend results about the representation of operators using frames from Ref.~\cite{xxlframoper1},
giving proofs about the invertibility and related properties of operators and the connected matrices.
Some of the new results have already been stated in \cite{xxlriek11} without proofs.

\section{Preliminaries and Notation}

We largely stick to the notation in \cite{xxlframoper1}. We will denote the (Moore-Penrose) pseudo-inverse of an operator $O$ by $O^\dagger$ (see e.g. \cite{ole1}).
Let us just remind the reader on the concept of frames:

\subsection{Frames}
For more details and proofs for this section refer e.g. to \cite{ole1,casaz1}.

A sequence ${\Psi} = \left( \psi_k | k \in K \right)$ 
 is called a \em frame \em for the Hilbert space $\Hil$, if constants $A,B > 0$ exist, such that 
\begin{equation} \label{sec:framprop1} A \cdot \norm{\Hil}{f}^2 \le \sum \limits_k \left| \left< f, \psi_k \right> \right|^2 \le B \cdot  \norm{\Hil}{f}^2  \ \forall \ f \in \Hil
\end{equation} 
Here $A$ is called a {\em lower} and $B$ an \em upper frame bound\em . 

For a Bessel sequence, $\Psi = ( \psi_k )$, let $C_{\Psi} : \Hil \rightarrow \ell^2 ( K )$ be the \em analysis  operator \em
$ C_{\Psi} ( f ) = \left( \left< f , \psi_k \right> \right)_k$. 
Let $D_{\Psi} : \ell^2( K ) \rightarrow \Hil $ be the  \em synthesis operator \em
$ D_{\Psi} \left( \left( c_k \right) \right) = \sum \limits_k c_k \cdot \psi_k $. 
Let $S_{\Psi} : \Hil  \rightarrow \Hil $ be the \em (associated) frame  operator \em
$ S_{\Psi} ( f  ) = \sum \limits_k  \left< f , \psi_k \right> \cdot \psi_k $. 
$C$ and $D$ are adjoint to each other, $D = C^*$ with $\norm{Op}{D} = \norm{Op}{C} \le \sqrt{B}$.  The series $\sum \limits_k c_k \cdot \psi_k$ converges unconditionally for all $(c_k) \in \ell^2$. 

For a frame ${\Psi} = ( \psi_k )$ with bounds $A,B$, $C$ is a bounded, injective operator with closed range and $S = C^*C = DD^*$ is a positive invertible operator satisfying $A I_\Hil \le S \le B I_\Hil$ and $B^{-1} I_\Hil \le S^{-1} \le A^{-1} I_\Hil$. Even more, we can find an expansion for every member of $\Hil$: 
The sequence $\tilde{\Psi} = \left( \tilde{\psi}_k \right) = \left( S^{-1} \psi_k \right)$ 
is a frame with frame bounds $B^{-1}$, $A^{-1} > 0$, the so called \em canonical dual frame\em .
 Every $f \in \Hil$ has the 
 expansions
$ f = \sum \limits_{k \in K} \left< f, \tilde{\psi}_k \right> \psi_k $
and 
$ f = \sum \limits_{k \in K} \left< f, \psi_k \right> \tilde{\psi}_k $
where both sums converge unconditionally in $\Hil$.

Two sequences $(\psi_k)$, $(\phi_k)$ are called \em biorthogonal \em 
 if 
$\left< \psi_k, \phi_j\right> = \delta_{kj}$ for all $h,j$.

A sequence $( \psi_k)$ in $\Hil$  is called a \em Riesz sequence \em if 
there exist constants $A$, $B >0$ such that the inequalities
$$ A \norm{2}{c}^2 \le \norm{\Hil}{\sum \limits_{k \in K} c_k \psi_k}^2 \le B \norm{2}{c}^2 $$
hold for all finite sequences $(c_k)$. It is called a {\em Riesz basis}, if it is complete as well.

For a frame $( \psi_k )$ the following conditions are equivalent:
\begin{enumerate}
\item[(i)] $( \psi_k )$ is a Riesz basis for $\Hil$.
\item[(ii)] The coefficients $( c_k ) \in \ell^2$ for the series expansion with $( \psi_k )$ are unique. So the synthesis operator $D$ is injective.
\item[(iii)] The analysis operator $C$ is surjective. 
\item[(iv)] $( \psi_k )$ and $( \tilde{\psi}_k )$ are biorthogonal.
\end{enumerate}
The {\em Gram matrix} $G_{\Psi, \Phi}$  is given by $\left( G_{\Psi, \Phi} \right)_{j,m} = \left< \phi_m , \psi_j \right>$, $j,m \in K$. Therefore as an operator form $\ell^2$ into $\ell^2$ $ G_{\Psi, \Phi} = C_{\Psi} \circ D_{\Phi}$. 
For a frame $G_{\Psi, \tilde \Psi}$ represents the projection on $\range{C_\Psi}$, denoted by $\Pi_{\range{C_\Psi}}$.

\subsection{Matrix representation of operators}
For orthonormal sequence it is well known, that operators can be uniquely described by a matrix representation \cite{gohberg1}. The same can be constructed with frames and their duals, see \cite{xxlframoper1}. 
Note that we will use the notation $\norm{\Hil_1 \rightarrow \Hil_2}{.}$ for the operator norm in $\BL(\Hil_1, \Hil_2)$ to be able to distinguish between different operator norms.
\begin{theorem} \label{sec:matbyfram1} Let $\Psi = (\psi_k)$ be a frame in $\Hil_1$ with bounds $A,B$, $\Phi = (\phi_k)$ in $\Hil_2$ with $A',B'$.
\begin{enumerate} \item Let $O : \Hil_1 \rightarrow \Hil_2$ be a bounded, linear operator. Then the infinite matrix 
$$ 
{\left( {\mathcal M}^{(\Phi , \Psi)} \left( O \right) \right)}_{m,n} = 
\left<O \psi_n, \phi_m \right>$$%
defines a bounded operator from $\ell^2$ to $\ell^2$ with $\norm{\ell^2 \rightarrow \ell^2}{\mathcal M} \le \sqrt{B \cdot B'} \cdot \norm{\Hil_1 \rightarrow \Hil_2}{O}$.
As an operator $\ell^2 \rightarrow \ell^2$ 
$$ {\mathcal M}^{(\Phi , \Psi)} \left( O \right) = C_{\Phi} \circ O \circ D_{\Psi} $$
This means the function ${\mathcal M}^{(\Phi , \Psi)} : \BL(\Hil_1,\Hil_2) \rightarrow \BL(\ell^2,\ell^2)$ is a well-defined bounded operator.
\item On the other hand let $M$ be an infinite matrix defining a bounded operator from $\ell^2$ to $\ell^2$, $\left(M c\right)_i = \sum \limits_k M_{i,k} c_k$. Then the operator $\mathcal{O}^{(\Phi , \Psi)}$ defined by 
$$ \left( \mathcal{O}^{(\Phi , \Psi)} \left( M \right)\right) h = \sum \limits_k  \left( \sum \limits_j M_{k,j} \left<h, \psi_j\right> \right) \phi_k \mbox{, for } h \in \Hil_1$$  
is a bounded operator from $\Hil_1$ to $\Hil_2$ with 
$$\norm{\Hil_1 \rightarrow \Hil_2}{\mathcal{O}^{(\Phi , \Psi)} \left( M \right)} \le \sqrt{B \cdot B'} \norm{\ell^2 \rightarrow \ell^2}{M}.$$
$$ \mathcal{O}^{(\Phi , \Psi)} (M) = D_{\Phi} \circ M \circ C_{\Psi} = \sum \limits_k  \sum \limits_j M_{k,j} \cdot \phi_k \otimes_i \overline{\psi}_j $$
This means the function ${\mathcal O}^{(\Phi , \Psi)} : \BL(\ell^2,\ell^2) \rightarrow \BL(\Hil_1,\Hil_2)$ is a well-defined bounded operator.
\end{enumerate} 
\end{theorem}

For frames more properties were proved\cite{xxlframoper1}: 
\begin{proposition} \label{sec:propmatropfram1} Let $\Psi = (\psi_k)$ be a frame in $\Hil_1$ with bounds $A,B$, $\Phi = (\phi_k)$ in $\Hil_2$ with $A',B'$. Then
\begin{enumerate}
\item $ \left( {\mathcal O^{(\Phi , \Psi)} \circ M^{(\tilde{\Phi}, \tilde{\Psi})}}\right)  = \identity{\BL(\Hil_1,\Hil_2)} = \left( {\mathcal O^{(\tilde{\Phi}, \tilde{\Psi})} \circ M^{(\Phi , \Psi)}}\right) $. \\
And therefore for all $O \in \BL(\Hil_1,\Hil_2)$:
$$ O = \sum \limits_{k,j} \left<O \tilde{\psi}_j, \tilde{\phi}_k \right>  \phi_k \otimes_i \overline{\psi}_j $$
\item  $\mathcal{M}^{(\Phi , \Psi)}$ is injective and $\mathcal{O}^{(\Phi , \Psi)}$ is surjective.
\item Let $\Hil_1 = \Hil_2$, then $\mathcal{O}^{(\Psi, \tilde{\Psi})} (Id_{\ell^2}) = \identity{\Hil_1}$ 
\item Let $\Xi = (\xi_k)$ be any frame in $\Hil_3$, and $O : \Hil_3 \rightarrow \Hil_2$ and $P: \Hil_1 \rightarrow \Hil_3$. Then
$$ \mathcal{M}^{(\Phi, \Psi)}\left( O \circ P \right) = \left( \mathcal{M}^{(\Phi, \Xi)}\left( O \right) \cdot \mathcal{M}^{(\tilde{\Xi}, \Psi)} \left( P \right) \right) $$
\end{enumerate}

\end{proposition}

\section{Properties of the Matrix Representation} \label{sec:descropfram0}

We can show some more connections of operators and their associated matrices:
\begin{proposition}\label{sec:proprep1} Let $\Phi$ and $\Psi$ be frames for $\Hil_1$ and $\Hil_2$ respectively.
Given $M \in \BL(l_2 , l_2)$, the following are equivalent:
\begin{enumerate}
\item[(i)] $\exists O \in \mathcal B (\Hil_1 , \Hil_2)$ such that $M=\Mat^{(\Phi,\Psi)}(O)$
\item[(ii)] $\exists M' \in \mathcal B (\ell_2, \ell_2) $ such that $ M=\Mat^{(\Phi,\Psi)} (\TOp (M'))$
\item[(iii)] $ \range{M} \subseteq \range{C_{\Phi}}$ and $\kernel{D_{\Psi}} \subseteq \kernel{M}$
\item[(iv)] $G_{\Phi, \tilde \Phi} \circ M \circ G_{\Psi, \tilde \Psi} = M$ 
\end{enumerate}
\end{proposition}
\begin{proof}
(i) $\Rightarrow$ (ii) because $\Op$ is surjective by Prop. \ref{sec:propmatropfram1}. \\
(ii) $\Rightarrow$ (i) is trivial. \\
(i) $\Rightarrow$ (iv): By Proposition~\ref{sec:propmatropfram1}~(1) 
$$ M=\Mat^{(\Phi,\Psi)}(O) \Rightarrow M = \Mat^{(\Phi,\Psi)} \left( \Op^{(\tilde \Phi,\tilde \Psi)} (M) \right) =  G_{\Phi, \tilde \Phi}\circ M \circ G_{\Psi, \tilde \Psi} .$$
(iv) $\Rightarrow$ (ii) As $G_{\Phi, \tilde \Phi}\circ M \circ G_{\Psi, \tilde \Psi} = \Mat^{(\Phi,\Psi)} \left( \tilde \Op^{(\tilde \Phi, \tilde\Psi)} (M) \right)$ we get (ii).\\
(iv) $\Leftrightarrow$ (iii) This is because $G_{\Phi, \tilde \Phi}\circ M \circ G_{\Psi, \tilde \Psi} = \Pi_{\range{C_\Phi}} ~ M ~ \Pi_{\range{C_\Psi}}$. 
\end{proof}

\subsection{Injectivity, Surjectivity and Invertibility}

In particular for solving operator or matrix equations the invertibility of the involved systems is of interest. We can show:

\begin{lemma}\label{prop:jectivityofO}
Let $M\in \mathcal B (\ell_2, \ell_2)$. Let $\Phi$ and $\Psi$ be frames for $\Hil_1$ and $\Hil_2$, respectively.
  \begin{enumerate}
    \item[(i)] If and only if $\Pi_{\range{C_{\Phi}}}~ M$ is injective on $\range{C_{\Psi}}$, then $\Op^{(\Phi,\Psi)}(M)$ is injective. In particular:
    \begin{enumerate}
    \item[(i')] If $\range{M|_{\range{C_{\Psi}}}} \subseteq \range{C_{\Phi}} $ and $M$ injective on $\range{C_{\Psi}}$, then $\Op^{(\Phi,\Psi)}(M)$ is injective.
    \end{enumerate}
    \item[(ii)] If and only if $\Pi_{\range{C_{\Phi}}} ~ M$ is surjective from $\range{C_{\Psi}}$ onto $\range{C_{\Phi}}$, then $\Op^{(\Phi,\Psi)} (M)$ is surjective.  In particular:
     \begin{enumerate}
    \item[(ii')] If $\range{M|_{\range{C_{\Psi}}}} = \range{C_{\Phi}} $, then $\Op^{(\Phi,\Psi)}(M)$ is surjective.
    \end{enumerate}
  \item[(iii)]  If and only if $\Pi_{\range{C_{\Phi}}} ~ M$ is bijective as an operator from $\range{C_{\Psi}}$ onto $\range{C_{\Phi}}$, $\Op^{(\Phi,\Psi)} (M)$ is bijective. In particular:
     \begin{enumerate}
    \item[(iii')] If $M$ is invertible from $\range{C_{\Psi}}$ onto $ \range{C_{\Phi}}$, then $\Op^{(\Phi,\Psi)}(M)$ is invertible.
    In this case 
    $$O^{-1}=\Op^{(\tilde\Psi,\tilde\Phi)}(M^\dagger)= 
\Op^{\Phi, \Psi}\left(G_{\tilde \Phi, \tilde \Psi} M^\dagger G_{\tilde \Phi,\tilde\Psi}\right) = S_{\tilde \Phi, \tilde \Psi} \Op^{\Phi, \Psi}\left(M^\dagger \right) S_{\tilde \Phi,\tilde\Psi} ,$$
where the pseudo-inverse $M^\dagger = \left\{\begin{array}{c c} M_{|_\range{C_{\Psi}}}^{-1} & \mbox{ on } \range{C_{\Psi}} \\ 0 & \mbox{ otherwise } \end{array} \right.$.
    \end{enumerate}
 \end{enumerate}
\end{lemma}
\begin{proof}
$D_\Psi$ is an invertible operator from $\range{C_{\Psi}}$ onto $\Hil$ and $C_\Phi$ from $\Hil$ onto $\range{C_{\Phi}}$. And because
$$ \Op^{(\Phi,\Psi)}(M) = D_\Phi M C_\Psi = D_\Phi ~ \Pi_{\ker{D_{\Phi}}^\perp} M C_\Psi = D_\Phi \Pi_{\range{C_{\Phi}}} M C_\Psi.$$
most of the results follow, immediately. 

Furthermore for (iii')
$$\Op^{(\Phi,\Psi)}(M) ~ \Op^{(\tilde\Psi,\tilde\Phi)}(M^{\dagger}) = D_\Phi ~ M ~ C_\Psi ~ D_{\tilde \Psi} ~ M^{\dagger} ~ C_{\tilde \Phi} = $$
$$ = D_\Phi M ~ \Pi_{\range{C_\Psi}} M^{-1} D_{\tilde \Phi} = D_\Phi M M^{-1} C_{\tilde \Phi} = D_\Phi C_{\tilde \Phi} = \identity{\Hil}. $$
As we know that $O$ is invertible, $O^{-1} =  \Op^{(\tilde\Psi,\tilde\Phi)}(M^{\dagger})$.

We can show
$$\Op^{\Phi, \Psi}\left(G_{\tilde \Phi \tilde \Psi} M^{\dagger} G_{\tilde \Phi, \tilde\Psi}\right) = \Op^{\Phi, \Psi}\left( C_{\tilde \Phi} D_{\tilde \Psi} M^{-1} C_{\tilde \Phi} D_{\tilde \Psi} \right) = $$
$$ =  D_\Phi C_{\tilde \Phi} D_{\tilde \Psi} M^{-1} C_{\tilde \Phi} D_{\tilde \Psi} C_\Psi = D_{\tilde \Phi} M^{-1} C_{\tilde \Psi} =  \Op^{(\tilde\Psi,\tilde\Phi)}(M^{-1}). $$

Finally
$$S_{\tilde \Psi, \tilde \Phi} \Op^{\Phi, \Psi}\left(M^\dagger \right) S_{\tilde \Psi,\tilde\Phi}v= D_{\tilde \Psi} C_{\tilde \Phi} D_{\Phi} M^{-1} C_\Psi D_{\tilde \Psi} C_{\tilde \Phi} = $$
$$ D_{\tilde \Psi} \Pi_{\range{C_\Phi}} M^{-1} C_\Psi D_{\tilde \Psi} C_{\tilde \Phi} = D_{\tilde \Psi} M^\dagger C_{\tilde \Phi} = \Op^{\widetilde \Phi, \widetilde \Psi}\left( M^\dagger \right).$$ 
\end{proof}

\begin{lemma}
Let $O\in \mathcal B (\Hil_1, \Hil_2)$ and $\Phi$ and $\Psi$ be frames for $\Hil_1$ and $\Hil_2$, respectively.
  \begin{enumerate}
    \item[(i)] If and only if $O$ is injective, $M=\Mat^{(\Phi,\Psi)}(O)$ is injective on $\range{C_{\Psi}}$.
    \item[(ii)] If and only if $O$ is surjective, $M=\Mat^{(\Phi,\Psi)}(O)$ is surjective from $\range{C_{\Psi}}$ onto $\range{C_{\Phi}}$.
    \item[(iii)] If and only if $O$ is bijective, $M=\Mat^{(\Phi,\Psi)}(O)$ is bijective as operator from $\range{C_{\Psi}}$ onto $\range{C_{\Phi}}$. In this case 
    $$M^{\dagger}= \Mat^{(\tilde \Psi, \tilde \Phi)} 
 (O^{-1})= G_{\tilde \Psi, \tilde \Phi}\circ \Mat^{(\Phi,\Psi)}\left(O^{-1}\right)
 G_{\tilde \Psi, \tilde \Phi} = \Mat^{(\Psi,\Phi)}\left(S_\Psi^{-1} O^{-1} S_\Phi^{-1}\right) . $$
    \end{enumerate}
\end{lemma}
\begin{proof}
With 
$ \Mat^{(\Phi,\Psi)} (O) = D_\Phi \, O \,  C_\Psi $
and the properties of $C_\Psi$ and $D_\Phi$ we have most of the wanted results like in the last proof.

For the inverse 
$$ \Mat^{(\Phi,\Psi)}(O) \circ \Mat^{(\tilde \Psi, \tilde \Phi)} (O^{-1}) = C_\Phi O D_\Psi C_{\tilde \Psi} O^{-1} D_{\tilde \Phi} =  \identity{\Hil}. $$

$$ G_{\tilde \Psi, \tilde \Phi}\circ \Mat^{(\Phi,\Psi)}(O^{-1})
 \circ G_{\tilde \Psi, \tilde \Phi} = C_{\tilde \Psi} D_{\tilde \Phi} C_\Phi O^{-1} D_\Psi C_{\tilde \Psi} D_{\tilde \Phi} =  $$
$$ =  C_{\tilde \Psi} O^{-1} D_{\tilde \Phi} = \Mat^{(\tilde \Psi, \tilde \Phi)} 
 (O^{-1}). $$
 
Finally
$$ \Mat^{(\tilde \Psi, \tilde \Phi)}  (O^{-1}) )=  C_{\tilde \Psi} O^{-1} D_{\tilde \Phi} = \Pi_{\range{C_\Psi}}  C_{\tilde \Psi} O^{-1} D_{\tilde \Phi} \Pi_{\kernel{D_{\tilde \Phi}}} = $$
$$ = C_{\Psi} D_{\tilde \Psi}  C_{\tilde \Psi} O^{-1} D_{\tilde \Phi} C_{\tilde \Phi} D_{\Phi} = \Mat^{(\Psi,\Phi)}\left(S_\Psi^{-1} O^{-1} S_\Phi^{-1}\right). $$
\end{proof}

\subsubsection{Invertibility and Riesz bases}

\begin{theorem}
Let $M$ be an infinite matrix defining a bounded operator from $\ell_2$ to $\ell_2$ and $\Phi$ and $\Psi$ be frames for a Hilbert~space~$\Hil$.
Then the following are equivalent:
\begin{itemize}
\item[(i)] For all matrices $\mathcal B (\ell_2, \ell_2)$ the following are equivalent:
    \begin{itemize}
      \item[(a)]$M$ is bijective on $l_2$.
      \item[(b)]$\Op^{(\Phi, \Psi)}(M)$ is bijective.
    \end{itemize}
\item[(ii)] Both, $\Phi$ and $\Psi$ are Riesz bases.
\end{itemize}
\end{theorem}
\begin{proof} (ii) $\Longrightarrow$ (i) 
This direction is clear by the properties of $\Mat$ and $\Op$ for Riesz bases (\cite{xxlframoper1} Theorem 3.5).

(i) $\Longrightarrow$ (ii)
Set $M = C_{\tilde \Phi} D_{\tilde \Psi}$, then $\Op^{(\Phi, \Psi)}(M) = D_\Phi C_{\tilde \Phi} D_{\tilde \Psi} C_\Psi = \identity{\Hil}$. By assumption $C_{\tilde \Phi} D_{\tilde \Psi}$ is bijective, in particular $D_{\tilde \Psi}$ is injective, therefore ${\tilde \Psi}$ and $\Psi$ are Riesz bases.
\end{proof}

\section{Decomposition}
For frames we know
$$ \mathcal{M}^{(\Phi, \Psi)}\left( O \circ P \right) = \left( \mathcal{M}^{(\Phi, \Xi)}\left( O \right) \cdot \mathcal{M}^{(\tilde{\Xi}, \Psi)} \left( P \right) \right). $$
Can similar properties for $\Op^{(\Phi, \Psi)}$ be shown?

\subsection{Decomposition and Riesz bases}

The following statement is the analogue of Proposition~\ref{sec:propmatropfram1}~(4) 
for $\mathcal O^{(\Phi, \Psi)}$. It provides conditions under which, also $\mathcal O^{(\Phi, \Psi)}$ is `well-behaved':
\begin{theorem}\label{prop:decompriesz}
Let $\Phi$, $\Xi$, and $\Psi$ be frames for $\Hil^1$, $\Hil^2$, and $\Hil^3$ resp., and $M^{(1)}$ and $M^{(2)}$ be infinite matrices defining bounded operators from $\ell_2$ to $\ell_2$. Then the following are equivalent:
\begin{itemize}
\item[(i)] $
  \mathcal O^{(\Phi, \Psi)}(M^{(1)} \cdot M^{(2)})
  =
  \mathcal O^{(\Phi, \Xi)}(M^{(1)}) \circ \mathcal O^{(\tilde \Xi, \Psi)}(M^{(2)})
  \quad \forall M^{(1)}, M^{(2)} \in \mathcal B(\ell_2, \ell_2)$
\item[(ii)] $\Xi$ is a Riesz basis.
\end{itemize}
\end{theorem}
\begin{proof}
  According to Theorem~\ref{sec:matbyfram1}~(2) the left hand side in $(i)$ equals
   $$LHS = D_{\Phi} \circ M^{(1)} \cdot M^{(2)} \circ C_{\Psi} \quad ,$$
   whereas the right hand side equals
   $$RHS = D_{\Phi} \circ M^{(1)} \circ C_{\Xi} \circ D_{\tilde \Xi} \circ M^{(2)} \circ C_{\Psi} \quad .$$
   
(ii) $\Longrightarrow$ (i)   If $\Xi$ is a Riesz sequence, $ C_{\Xi} \circ D_{\tilde \Xi} = \identity{l_2}$, and so (i) is fullfilled.
   
(i) $\Longrightarrow$ (ii) Set $M^{(1)} := C_{\tilde \Phi} D_E$ and $M^{(2)} = C_E D_{\tilde \Psi}$, for any ONB $E$. Then 
$$ LHS = D_{\Phi} ~  C_{\tilde \Phi} ~ D_E ~  C_E ~ D_{\tilde \Psi}  ~ C_{\Psi} = \identity{l_2},  $$
and 
$$ RHS =   D_{\Phi} ~  C_{\tilde \Phi} ~ D_E ~ C_{\Xi} ~ D_{\tilde \Xi} ~  C_E ~ D_{\tilde \Psi}  ~  C_{\Psi} = D_E ~ C_{\Xi} ~ D_{\tilde \Xi} ~  C_E. $$
So $D_E ~ C_{\Xi} ~ D_{\tilde \Xi} ~  C_E = \identity{l_2}$, and therefore $C_{\Xi} ~ D_{\tilde \Xi} = \identity{l_2}$
 and $\Xi$ is a Riesz basis.
   
\end{proof}

\begin{remark}
More generally, we could also ask for a decomposition using two frames $\Xi^{(1)}$ and $\Xi^{(2)}$ for $\Hil^2$ instead of $\Xi$ and $\tilde \Xi$ respectively, leading to the condition that $\Xi^{(1)}$ and $\Xi^{(2)}$ are biorthogonal, and therefore get back to the assumptions in the theorem. 
\end{remark}

Applying this result to frame multipliers, which corresponds to operators induced by diagonal matrices, we get \cite{xxlmult1} Prop.7.4.

If, however, we tune the properties of the frame $\Xi$ to the properties of a specific pair of matrices $M^{(1)}$ and $M^{(2)}$, we can prove a corresponding result for the general frame case.

\subsection{Decomposition and frames}

\begin{theorem}\label{prop:decompframes}
Let $\Phi$, $\Xi$, and $\Psi$ be frames for $\Hil^1$, $\Hil^2$, and $\Hil^3$ respectively, and $M^{(1)}$ and $M^{(2)}$ be infinite matrices defining bounded operators from $\ell_2$ to $\ell_2$.
Then if
\begin{itemize}
\item[(a)] $ M^{(2)}(\range{C_{\Psi}}) \subseteq \range{C_{\Xi}} $~, or
\item[(b)] $\left(\kernel{ D_{\Phi} \circ M^{(1)}}\right)^{\bot} \subseteq \range{C_{\Xi}} $~,
\end{itemize}
we get that $
  \mathcal O^{(\Phi, \Psi)}(M^{(1)} \cdot M^{(2)})
  =
  \mathcal O^{(\Phi, \Xi)}(M^{(1)}) \circ \mathcal O^{(\tilde \Xi, \Psi)}(M^{(2)})$~.
\end{theorem}
\begin{proof} Since $\mathcal O^{(\Phi, \Xi)}(M^{(1)}) \circ \mathcal O^{(\tilde \Xi, \Psi)}(M^{(2)}) = D_{\Phi} \circ M^{(1)} \circ C_{\Xi} \circ D_{\tilde \Xi} \circ M^{(2)} \circ C_{\Psi} $ and $C_{\Xi} \circ D_{\tilde \Xi} = \Pi_{\range{C_{\Xi}}}$, we see that if $ M^{(2)}(\range{C_{\Psi}}) \subseteq \range{C_{\Xi}} $ it holds that $C_{\Xi} \circ D_{\tilde \Xi} \circ M^{(2)} \circ C_{\Psi} = M^{(2)} \circ C_{\Psi}$ and therefore the statements with assumption $(a)$ is true.\\
If, on the other hand, we assume that $\left(\kernel{ D_{\Phi} \circ M^{(1)}}\right)^{\bot} \subseteq \range{C_{\Xi}} $ then we get $D_{\Phi} \circ M^{(1)} \circ C_{\Xi} \circ D_{\tilde \Xi} = D_{\Phi} \circ M^{(1)}$, which finishes the proof.
\end{proof}

\begin{remark}
Again, instead of using a single frame and its dual for the decomposition, we could look at a pair of frames $\Xi^{(1)}$ and $\Xi^{(2)}$. This would result in the assumptions for Theorem~\ref{prop:decompframes} to be that
\begin{itemize}
\item[(a)] $C_{\Xi^{(1)}}D_{\Xi^{(2)}}|_{\range{M^{(2)} \circ C_{\Psi}}} = \identity{\range{M^{(2)} \circ C_{\Psi}}}$~, or
\item[(b)] $C_{\Xi^{(1)}}D_{\Xi^{(2)}}|_{\left(\kernel{ D_{\Phi} \circ M^{(2)}}\right)^{\bot}}=\identity{\left(\kernel{ D_{\Phi} \circ M^{(2)}}\right)^{\bot}} $~.
\end{itemize}
\end{remark}

\section{Summary and Outlook}

We have shown some basic properties of frame representations of operators, in particular with regard to their invertibility.

In the future, we are planning to investigate the relation of the operator representation using frames presented here with special focus on the finite section method and localized frames.
Furthermore we will apply this concept to the numerical solution of the Helmholtz equation using wavelet frames.

\subsection*{\bf Acknowledgments} 
The work on this paper was partly supported by the WWTF project MULAC ('Frame Multipliers:
Theory and Application in Acoustics; MA07-025) and by the Austrian Science
Fund (FWF) START-project FLAME ('Frames and Linear Operators for Acoustical Modeling and
Parameter Estimation'; Y 551-N13).

The authors are thankful to W. Kreuzer and J.-P. Antoine for valuable comments and discussions.

\providecommand{\bysame}{\leavevmode\hbox to3em{\hrulefill}\thinspace}
\providecommand{\MR}{\relax\ifhmode\unskip\space\fi MR }
\providecommand{\MRhref}[2]{%
  \href{http://www.ams.org/mathscinet-getitem?mr=#1}{#2}
}
\providecommand{\href}[2]{#2}

\end{document}